\documentclass{amsart}
\usepackage{amsfonts,amsmath,amsthm,amscd,amssymb,latexsym,}

\newtheorem{theorem}{Theorem}
\newtheorem{lemma}{Lemma}
\newtheorem{corollary}{Corollary}
\DeclareMathOperator{\arccot}{arccot}
\theoremstyle{definition}
\newtheorem{example}{Example}
\newtheorem{remark}{Remark}
\begin{document}

\title[Trigonometric Polynomial and Stabilization of Chaos ]{Extremal trigonometric polynomials and the problem of optimal stabilization of chaos}
\author{ D.V. Dmitrishin and A.D. Khamitova}

\begin{abstract}
For a pair of conjugate trigonometric polynomials $C(t)=\sum\limits_{j=1}^na_j\cos jt,$
$S(t)=\sum\limits_{j=1}^na_j\sin jt$ normalized by the condition \mbox{$\sum\limits_{j=1}^na_{j}=1$}
the  extremal problem
$$
\mathop{\sup } \limits_{a_{1} ,\ldots,a_{n}}
\min_{t}\{C(t): S(t)=0\}=-\tan^{2} \frac{\pi }{2(n+1)}
$$
is solved.
The solution is given by a unique non-negative extremal Fejer polynomial.
An application of this result in the control theory for nonlinear descrete systems is shown.
This paper is an extended version of \cite{DK}
\end{abstract}



\vskip 10 pt
\maketitle

\bigskip
As we know, the problem of optimal impact on chaotic regimes is one of the most fundamental in nonlinear dynamics ~\cite{1}.
For a multiparameter families of discrete systems this problem can be reduced to the choice of a direction that provides maximum stability in one-parameter space.  When we change this parameter we can explore the sequence of bifurcations that leads to occurrence of a chaotic attractor.

\bigskip
The first bifurcation values of the parameter correspond to the loss of stable equilibrium position in the system. These values are related to the area of Schur's stability of family of polynomials,
\begin{equation}\label{(1)}
\left\{f(\lambda)=\lambda^n+k\left(a_1\lambda^{n-1}+\dots+a_n\right),\
\sum_{j=1}^na_j=1\right\}
\end{equation}
in the space of parameters $k$.  All polynomials of the family ~(\ref{(1)}) are stable for $k=0$. Because of continuous dependence of the zeros of polynomials on their coefficients
 there exist two positive constants $k_1,k_2$, which depend on the coefficients  $a_1,\dots,a_n$, such that for $k\in\left(-k_1,k_2\right)$ the family ~(\ref{(1)}) is
still stable and the stability fails if $k=k_{2} +\varepsilon $ or $k=-k_{1} -\varepsilon $ for positive $\varepsilon$ small enough.
\medskip

We need to maximize the length of robust stability segment i.e. the function
\begin{equation}\label{(2)}
\Phi\left(a_1,\dots,a_n\right)=
k_1\left(a_1,\dots,a_n\right)+
k_2\left(a_1,\dots,a_n\right).
\end{equation}

The function ~(\ref{(2)}) has simple geometrical meaning. Since
$$
\frac{f\left({\rm e}^{it}\right)}{k\cdot{\rm e}^{int}}=
\frac1k+\sum\limits_{j=1}^na_j\cos jt-i\cdot\sum\limits_{j=1}^na_j\sin jt,
$$
the points of intersection of the curve
$\left\{x=\sum\limits_{j=1}^na_j\cos jt,\ y=-\sum\limits_{j=1}^na_j\sin jt\right\}$
on the $OXY$ plane with $OX$ axis correspond to those values of parameter $k$, for which the polynomial of the family ~(\ref{(1)}) has zeros on the unit circle. The length of the longest segment which is defined by these points of intersection is $\frac{1}{k_{1} } +\frac{1}{k_{2} }.$

Since
$\sum\limits_{j=1}^na_j=1$,
then  $f(1)=1+k$, therefore
$\max\limits_{a_1,\dots,a_n}\left\{k_1\left(a_1,\dots,a_n\right)\right\}\le1$.
Let
$$I=\sup\limits_{a_1,\dots,a_n}\min\limits_t\left\{C(t):
\ S(t)=0,\ \sum\limits_{j=1}^na_j=1\right\},$$
where
$C(t)=\sum\limits_{j=1}^na_j\cos jt$,
$S(t)=\sum\limits_{j=1}^n\sin jt$.
Then
$\max\limits_{a_1,\dots,a_n}\left\{k_2\left(a_1,\dots,a_n\right)\right\}\le-1/I$.
Hence,
$$\max\limits_{a_1,\dots,a_n}\Phi\left(a_1,\dots,a_n\right)\le1-1/I.$$
Note, that the value of $I$ is negative as it follows from the  Lemma \ref{lem1}.

The main result of this article is Theorem~\ref{teo1}, which states that
$$I=-\tan^2\frac\pi{2(n+1)}.$$
Therefore,
\begin{equation}\label{(9)}
\max_{a_1,\dots,a_n}\left\{k_2\left(a_1,\dots,a_n\right)\right\}=
\cot^2\frac\pi{2(n+1)}.
\end{equation}
Theorem~\ref{teo1} implies that the extremal coefficients $a_1^0,\dots,a_n^0$ in the problem~(\ref{(9)}) are positive.
Hence
$$
k_1\left(a_1^0,\dots,a_n^0\right)=1.
$$
Finally,
$$
\max_{a_1,\dots,a_n}\Phi\left(a_1,\dots,a_n\right)=
1+\cot^2\frac\pi{2(n+1)}=
\frac1{\sin^2\frac\pi{2(n+1)}}.
$$

\bigskip
In order to find the value of $I$, we need an auxiliary results that are included in the following five lemmas and several corollaries.

Let
\begin{equation}\label{(3)}
\rho=\min_t\left\{C(t):\ S(t)=0,\,\sum_{j=1}^na_j=1\right\}.
\end{equation}
Since the polynomials $C(t)$ and $S(t)$ are periodic, it is enough to consider the minimum in~(\ref{(3)}) on the segment $ \left[ 0, \pi \right]$.

\begin{lemma}\label{lem1}
The value of $\rho $ is negative.
\end{lemma}

\begin{proof}
Let  $F(z)=\sum _{j=1}^{n}a_{j} z^{j}$, where $z$ is a complex variable. It is clear that $z_{0} =0$ is a zero of $F(z)$.
Because of continuous dependence of polynomial zeros on the coefficients the function
$$
F_\epsilon(z)=\epsilon+\sum _{j=1}^{n}a_{j} z^{j}
$$
has a zero $z_\epsilon$ of the module less then one for small enough $\epsilon.$

Now, let $\rho\ge 0.$ Then for any positive $\epsilon$ the graph of the function $x+iy=F_\epsilon(e^{it})$ does not intersect the negative part of the real axis in $(x,y)$ plane, i.e. does not surround the origin for $t\in[0,2\pi].$ By the argument principe the function $F_\epsilon(z)$ does not have zeros inside the unit disc which is false. By the contrapositive argument  Lemma~\ref{lem1} is proved.

\end{proof}

\begin{lemma}\label{lem2}
Let
$S\left(t_1\right)=0$, $t_1\in(0,\pi)$. Then trigonometric polynomials $S(t)$ and $C(t)$
can be written in a unique way as
$$S(t)=\left(\cos t-\cos t_1\right)\cdot\sum\limits_{j=1}^{n-1}a_j^{(1)}\sin jt,\quad C(t)=-\frac{a_1^{(1)}}2+\left(\cos t-\cos t_1\right)\cdot
\sum\limits_{j=1}^{n-1}a_j^{(1)}\cos jt.$$
\end{lemma}

\begin{proof}
If the coefficients $a_1\dots,a_n$ and $a_1^{(1)},\dots,a_{n-1}^{(1)}$  are connected by the relations

\begin{equation}\label{(4)}
\left\{
\begin{array}{llllll}
a_1=-\cos t_1\cdot a_1^{(1)}+\frac12a_2^{(1)},\\[0.1cm]
a_2=\frac12a_1^{(1)}-\cos t_1\cdot a_2^{(1)}+\frac12a_3^{(1)},\\
\dots\\
a_{n-2}=\frac12a_{n-3}^{(1)}-\cos t_1\cdot a_{n-2}^{(1)}+\frac12 a_{n-1}^{(1)},\\[0.1cm]
a_{n-1}=\frac12 a_{n-2}^{(1)}-\cos t_1\cdot a_{n-1}^{(1)},\\[0.1cm]
a_n=\frac12a_{n-1}^{(1)},\\
\end{array}
\right.
\end{equation}

then

\begin{multline}\nonumber
 S(t)=\sum_{j=1}^na_j\sin jt=\left(-\cos t_1\cdot a_1^{(1)}+\frac12a_2^{(1)}\right)\sin t+
 \dots+\\
+ \left(\frac12a_{n-3}^{(1)}-\cos t_1\cdot a_{n-2}^{(1)}+
 \frac12a_{n-1}^{(1)}\right)\sin(n-2)t+\\
 +\left(\frac12a_{n-2}^{(1)}-\cos t_1\cdot a_{n-1}^{(1)}\right)\cdot \sin(n-1)t+
 \frac12a_{n-1}^{(1)}\sin nt=\\
 =\frac12\sum_{j=1}^{n-1}a_j^{(1)}\sin(j+1)t-\cos t_1\cdot\sum_{j=1}^{n-1}a_j^{(1)}\sin jt+
 \frac12\sum_{j=2}^{n-1}a_j^{(1)}\sin(j-1)t=\\
 =\frac12\sum_{j=1}^{n-1}a_j^{(1)}(\sin(j+1)t+\sin(j-1)t)-
 \cos t_1\cdot\sum_{j=1}^{n-1}a_j^{(1)}\sin jt=\\
 =\left(\cos t-\cos t_1\right)\cdot\sum_{j=1}^{n-1}a_j^{(1)}\sin jt.
 \end{multline}

The determinant of the system of the last  $n-1$  equations is equal to $2^{1-n}$, therefore the coefficients
 $a_1^{(1)},\dots,a_{n-1}^{(1)}$ can be expressed in a unique way through the coefficients $a_2,\dots,a_n$,
and the representation
 $S(t)=\left(\cos t-\cos t_1\right)\cdot\sum\limits_{j=1}^{n-1}a_j^{(1)}\sin jt$
 is unique. Similarly, the system~(\ref{(4)}) implies
$$
 C(t)=\sum_{j=1}^na_j\cos jt=\left(-\cos t_1\cdot a_1^{(1)}+\frac12a_2^{(1)}\right)\cos t+
 \dots+
$$
$$
\left(\frac12a_{n-3}^{(1)}-\cos t_1\cdot a_{n-2}^{(1)}+
 \frac12a_{n-1}^{(1)}\right)\cos(n-2)t+
$$
$$
\left(\frac12a_{n-2}^{(1)}-\cos t_1\cdot a_{n-1}^{(1)}\right)\cdot \cos(n-1)t+
 \frac12a_{n-1}^{(1)}\cos nt=
$$
$$
\frac12\sum_{j=1}^{n-1}a_j^{(1)}\cos(j+1)t-\cos t_1\cdot\sum_{j=1}^{n-1}a_j^{(1)}\cos jt+
 \frac12\sum_{j=2}^{n-1}a_j^{(1)}\cos(j-1)t=
$$
$$
\frac12\sum_{j=1}^{n-1}a_j^{(1)}(\cos(j+1)t+\cos(j-1)t)-
 \cos t_1\cdot\sum_{j=1}^{n-1}a_j^{(1)}\cos jt=
$$
$$
-\frac{a_1^{(1)}}2+\left(\cos t-\cos t_1\right)\cdot\sum_{j=1}^{n-1}a_j^{(1)}\cos jt.
$$

Conversely,  the equations
 $S(t)=\left(\cos t-\cos t_1\right)\cdot\sum\limits_{j=1}^{n-1}a_j^{(1)}\sin jt$ and
  $C(t)=-\frac{a_1^{(1)}}2+\left(\cos t-\cos t_1\right)\cdot
  \sum\limits_{j=1}^{n-1}a_j^{(1)}\cos jt$ imply ~(\ref{(4)}).
 \end{proof}

\begin{remark}\label{rem1}
Let the trigonometric polynomial $S(t)$  be represented in the form
$$
S(t)=\left(\cos t-(-1)^q\right)\cdot\sum_{j=1}^{n-1}a_j^{(1)}\sin jt,
$$
where  $q$  is an integer. Then $t_1=\pi q$  is a multiple root for the polynomial $S(t)$.
\end{remark}

\begin{corollary}\label{cor1}
If $S\left(t_1\right)=0$ and $t_1\in(0,\pi)$, then $C\left(t_1\right)=-\frac{a_1^{(1)}}2$,
where  $a_1^{(1)}$  is determined in a unique way through the coefficients $a_2,\dots,a_n$ from the system~(\ref{(4)}).
\end{corollary}

\begin{corollary}\label{cor2}
If
$S\left(t_1\right)=0$ and $t_1\in(0,\pi)$, then
$C(\pi)=-\frac{a_1^{(1)}}2-\left(1+\cos t_1\right)\cdot
\sum\limits_{j=1}^{n-1}(-1)^j\cdot a_j^{(1)}$.
\end{corollary}

\begin{lemma}\label{lem3}
Let
$S\left(t_1\right)=\dots=S\left(t_m\right)=0$,
$t_j\in(0,\pi)$, $j=1,\dots,m$ $(m<n)$.
Then the trigonometric polynomial $S(t)$ is presented uniquely by
$$ S(t)=\prod_{j=1}^m\left(\cos t-\cos t_j\right)\cdot\sum_{j=1}^{n-m}a_j^{(m)}\sin jt.$$
\end{lemma}

\begin{proof}
By Lemma~\ref{lem2},
\begin{multline}\nonumber
S(t)=\left(\cos t-\cos t_1\right)\cdot\sum_{j=1}^{n-1}a_j^{(1)}\sin jt=\\
=\prod_{j=1}^2\left(\cos t-\cos t_j\right)\cdot\sum_{j=1}^{n-2}a_j^{(2)}\sin jt=
\dots
=\prod_{j=1}^m\left(\cos t-\cos t_j\right)\cdot\sum_{j=1}^{n-m}a_j^{(m)}\sin jt.
\end{multline}
The desired representation is will be obtained if the coefficients
$a_1,\dots,a_n$ and
$a_1^{(m)},\dots,a_{n-m}^{(m)}$
are related via the collection of $m$ systems of equations consisting of the system ~(\ref{(4)}) and the following $m-1$ systems
\begin{equation}\label{(5)}
\left\{
\begin{array}{llllll}
a_1^{(j-1)}=-\cos t_j\cdot a_1^{(j)}+\frac12a_2^{(j)},\\[0.1cm]
a_2^{(j-1)}=\frac12a_1^{(j)}-\cos t_j\cdot a_2^{(j)}+\frac12a_3^{(j)},\\
\dots\\
a_{n-j-1}^{(j-1)}=\frac12a_{n-j-2}^{(j)}-\cos t_j\cdot a_{n-j-1}^{(j)}+\frac12 a_{n-j}^{(j)},\\[0.1cm]
a_{n-j}^{(j-1)}=\frac12 a_{n-j-1}^{(j)}-\cos t_j\cdot a_{n-j}^{(j)},\\[0.1cm]
a_{n-j+1}^{(j-1)}=\frac12a_{n-j}^{(j)},\\
\end{array}
\right.
\end{equation}
$j=2,\dots,m$.
The coefficients $a_1^{(1)},\dots,a_{n-1}^{(1)}$ are uniquely determined by the coefficients $a_2,\dots,a_n$  from the system~(\ref{(4)}).
 Moreover, the coefficients $a_1^{(j)},\dots,a_{n-j}^{(j)}$, $j=2,\dots,m$ are uniquely determined by the coefficients  $a_2^{(j-1)},\dots,a_{n-j+1}^{(j-1)}$
of the $j$-th system from the collection~(\ref{(5)}).

 \end{proof}

 \begin{lemma}\label{lem4}
Let
\begin{equation}\label{(6)}
S\left(t_1\right)=\dots=S\left(t_m\right)=0,
C\left(t_1\right)=\dots=C\left(t_m\right),\quad m<n,
\end{equation}
where $t_1,\dots,t_m$  are pairwise distinct and belong to the interval $(0,\pi)$. Then $C(t)$  is uniquely represented by
\begin{equation}\label{(7)}
  C(t)=-\frac{a_m^{(m)}}{2^m}+\prod_{j=1}^m\left(\cos t-\cos t_j\right)
  \cdot\sum_{j=m}^{n-m}a_j^{(m)}\cos jt.
\end{equation}
\end{lemma}

\begin{proof}
First, notice that  $n\ge2m$. Since $t_j\in(0,\pi)$, $j=1,\dots,m$, then identities ~(\ref{(6)}) are equivalent of the
existence of $2m$ roots for the algebraic polynomial
$f(z)=a_0+\sum\limits_{j=1}^na_j\cdot z^j$ on the unit circles ${\rm e}^{it_j}$, ${\rm e}^{-it_j}$, $j=1,\dots,m$.
Therefore,  $2m\le n$.

We prove~(\ref{(7)}) by induction.
 Lemma~\ref{lem2} implies the expression for  $m=1$. Applying Lemma~\ref{lem2} twice, we get
\begin{multline}\nonumber
C(t)=-\frac{a_1^{(1)}}2+\left(\cos t-\cos t_1\right)\cdot
\left(-\frac{a_1^{(2)}}2+
\left(\cos t-\cos t_2\right)
\sum_{j=1}^{n-2}a_j^{(2)}\cos jt\right)=\\
=
-\frac{a_1^{(1)}}2-\frac{a_1^{(2)}}2\left(\cos t-\cos t_1\right)+
\left(\cos t-\cos t_1\right)\cdot\left(\cos t-\cos t_2\right)\cdot
\sum_{j=1}^{n-2}a_j^{(2)}\cos jt.
\end{multline}
Since
$C\left(t_1\right)=C\left(t_2\right)$  and
$\cos t_1\ne\cos t_2$, then  $a_1^{(2)}=0$.
Having in mind the first equation of the first system from the collection~(\ref{(5)}) we get
\begin{multline}\nonumber
C(t)=-\frac12\left(-\cos t_2\cdot a_1^{(2)}+\frac12a_1^{(2)}\right)+
\left(\cos -\cos t_1\right)\cdot\left(\cos t-\cos t_2\right)\cdot
\sum_{j=2}^{n-2}a_j^{(2)}\cos jt=\\
=-\frac14a_2^{(2)}+\left(\cos t-\cos t_1\right)\cdot\left(\cos t-\cos t_2\right)\cdot
\sum_{j=2}^{n-2}a_j^{(2)}\cos jt.
\end{multline}
Hence,~(\ref{(7)}) is valid for $m=2$.
\medskip

Assuming that ~(\ref{(7)}) is valid  $m-1$ let check it for $m$:
\begin{multline}\nonumber
C(t)=-\frac{a_{m-1}^{(m-1)}}{2^{m-1}}+\prod_{j=1}^{m-1}
\left(\cos -\cos t_j\right)\cdot
\sum_{j=m-1}^{n-m+1}a_j^{(m-1)}\cos jt=\\
=-\frac{a_{m-1}^{(m-1)}}{2^{m-1}}+\prod_{j=1}^{m-1}\left(\cos t-\cos t_j\right)\cdot
\left(-\frac{a_1^{(m)}}2+\left(\cos t-\cos t_m\right)\cdot
\sum_{j=1}^{n-m}a_j^{(m)}\cos jt\right).
\end{multline}
Since  $C\left(t_1\right)=C\left(t_m\right)$ and
$\cos t_m\ne\cos t_1,\dots,\cos t_m\ne\cos t_{m-1}$, then $a_1^{(m)}=0$.
Let consider the last system from the collection~(\ref{(5)}) and recall that
$a_1^{(m-1)}=\dots=a_{m-2}^{(m-1)}=0$. Then the first equation implies that  $a_2^{(m)}=0$,
and subsequently the second, the third and all the others equations implies
$a_3^{(m)}=0,\dots,a_{m-1}^{(m)}=0$. Therefore, taking into account the first equation of the last system from the collection~(\ref{(5)})
we obtain:
\begin{multline}\nonumber
C(t)=-\frac1{2^{m-1}}\cdot
\left(\frac12a_{m-2}^{(m)}-\cos t_m\cdot a_{m-1}^{(m)}+\frac12a_m^{(m)}\right)+\\
+\prod_{j=1}^{m-1}\left(\cos t-\cos t_j\right)\cdot
\left(-\frac{a_1^{(m)}}2+\left(\cos t-\cos t_m\right)\cdot
\sum_{j=1}^{n-m}a_j^{(m)}\cos jt\right)=\\
=-\frac{a_m^{(m)}}{2^m}+\prod_{j=1}^m\left(\cos t-\cos t_j\right)\cdot
\sum_{j=m}^{n-m}a_j^{(m)}\cos jt.
\end{multline}
\end{proof}

 \begin{corollary}\label{cor3}
If~(\ref{(6)}) holds then
$$
C\left(t_1\right)=\dots=C\left(t_m\right)=-\frac{a_m^{(m)}}{2^m},
$$
where  $a_m^{(m)}$ are uniquely determined via the coefficients  $a_2,\dots,a_n$
from the collection of the systems of equations~(\ref{(4)}),~(\ref{(5)}).
\end{corollary}

 \begin{corollary}\label{cor4}
If~(\ref{(6)}) holds then
$$
C\left(\pi\right)=-\frac{a_m^{(m)}}{2^m}
-\prod_{j=1}^m\left(1+\cos t_j\right)\cdot\sum_{j=m}^{n-m}(-1)^ja_j^{(m)}.
$$
\end{corollary}

 \begin{lemma}\label{lem5}
Let
$$
S(t)=\prod_{j=1}^m\left(\cos t-\cos t_j\right)\cdot\sum_{j=m}^{n-m}a_j^{(m)}\sin jt,
$$
where $t_1,\dots,t_m$  are pairwise distinct and belong to the interval $(0,\pi)$ while
$2\le m\le\frac n2$.
Then $C\left(t_1\right)=\dots=C\left(t_m\right)$.
\end{lemma}

\begin{proof}
From the numbers $t_1,\dots,t_m$ let chose any, say $t_1$.
Then by Lemma~\ref{lem2},
$S(t)=\left(\cos t-\cos t_1\right)\cdot\sum\limits_{j=1}^{n-1}a_j^{(m)}\sin jt$.
Since  $a_1^{(m)}=\dots=a_{m-1}^{(m)}=0$,
then from the systems~(\ref{(5)}) subsequently find
$a_1^{(m-1)}=\dots=a_{m-2}^{(m-1)}=0$,
$a_{m-1}^{(m-1)}=\frac{a_m^{(m)}}2$,
$a_1^{(m-2)}=\dots=a_{m-3}^{(m-2)}=0$,
$a_{m-2}^{(m-2)}=\frac{a_{m-1}^{(m-1)}}2,\ \dots,\ a_1^{(1)}=\frac{a_2^{(2)}}2$.
From here
$a_1^{(1)}=\frac{a_m^{(m)}}{2^{m-1}}$. By Lemma~\ref{lem2},
$C\left(t_1\right)=-\frac{a_1^{(1)}}2=-\frac{a_m^{(m)}}{2^m}$.
The coefficient $a_m^{(m)}$  is independent on $t_1,\dots,t_m$, therefore
$C\left(t_j\right)=-\frac{a_m^{(m)}}{2^m}$, $j=1,\dots,m$.
\end{proof}

 \begin{theorem}\label{teo1}
 Let $C(t)$ and $S(t)$  be a pair of conjugated trigonometric polynomials
 $$
 C(t)=\sum_{j=1}^na_j\cos jt,\quad S(t)=\sum_{j=1}^na_j\sin jt,
 $$
normalized by the conditions $\sum\limits_{j=1}^na_j=1$.

 Let  $I$  be a solution to the extremal problem
 $\sup\limits_{a_1,\dots,a_n}\min\limits_t\left\{C(t):\ S(t)=0\right\}$.
 Then
 $$
 I=-\tan^2\frac\pi{2(n+1)}.
 $$
 \end{theorem}

 \begin{proof}
 The function $S(t)$ is zero at  $t=\pi$  for any coefficients $a_1,\dots,a_n$.
 The value of $\sup\left\{\rho\left(a_1,\dots,a_n\right)\right\}$
will be found searching over the set
$$
A_R=\left\{\left(a_1,\dots,a_n\right):\
\sum\limits_{j=1}^n=1,\,\sum\limits_{j=1}^n\left|a_j\right|\le R\right\}.
$$
The function $\rho\left(a_1,\dots,a_n\right)$ is continuous on the set $A_R$,
except the points  $\left(a_1,\dots,a_n\right)$, for which minimal value $C(t)$
achieves at those zeros of $S(t)$ where it does not change sign. A lower limit of the function $\rho\left(a_1,\dots,a_n\right)$
is equal the value of the function at the points of discontinuity which means that the function $\rho\left(a_1,\dots,a_n\right)$
is semi-continuous from below.
\medskip

Together with the function $\rho\left(a_1,\dots,a_n\right)$ we will consider the function
$$
\rho_1\left(a_1,\dots,a_n\right)=\min_{t\in[0,\pi]}\left\{C(t):\ t\in T\cup\{\pi\}\right\},
$$
where $T$ is a set inside the interval $(0,\pi)$, where the function $S(t)$ changes sign.
The set $T\cup\{\pi\}$ is a subset of all zeros of  $S(t)$, therefore
$\overline{\rho}\le\overline{\rho_1}$, where
$\overline{\rho}\le\sup\limits_{\left(a_1,\dots,a_n\right)\in A_R}
\left\{\rho\left(a_1,\dots,a_n\right)\right\}$ and
$\overline{\rho_1}\le\sup\limits_{\left(a_1,\dots,a_n\right)\in A_R}
\left\{\rho_1\left(a_1,\dots,a_n\right)\right\}$.

The function $\rho_1\left(a_1,\dots,a_n\right)$ is upper semi-continuous therefore it achieves the maximum value on the set $A_R$
i.e. $\overline{\rho_1}=\max\limits_{\left(a_1,\dots,a_n\right)\in A_R}
\left\{\rho_1\left(a_1,\dots,a_n\right)\right\}$. Lemma 1 implies that $\rho_1<0.$

The pair of the trigonometric polynomials $\left\{C^0(t),\,S^0(t)\right\}$, on which the maximum is achieved will be called the optimal pair.

Let for the optimal polynomial $S^0(t)$ the set $T=\left\{t_1,\dots,t_q\right\}$, $0\le q\le n-1$ be nonempty.
Let
$$
\min\left\{C^0\left(t_1\right),\dots,C^0\left(t_q\right)\right\}=C^0\left(t_1\right),
$$
assuming additionally that
$C^0\left(t_1\right)=C^0\left(t_j\right)$, $j=1,\dots,m$ $(1\le m\le q)$,
$C^0\left(t_1\right)<C^0\left(t_j\right)$,
$j=m+1,\dots, q$.

Then two cases are possible: either $C^0\left(t_1\right)<C^0(\pi)$ or
$C^0(\pi)\le C^0\left(t_1\right)$.

{\it Case 1.}
By Lemmas~\ref{lem3},~\ref{lem4},
 the trigonometrical polynomials $S^0(t)$, $C^0(t)$ have form
$$
S^0(t)=\prod_{j=1}^m\left(\cos t-\cos t_j\right)\cdot
\sum_{j=m}^{n-m}a_j^{(m)}\sin jt,
$$
$$
C^0(t)=-\frac{a_m^{(m)}}{2^m}+\prod_{j=1}^m\left(\cos t-\cos t_j\right)\cdot
\sum_{j=m}^{n-m}a_j^{(m)}\cos jt.
$$
Since  $C^0\left(t_1\right)=-\frac{a_m^{(m)}}{2^m}$ and  $\overline{\rho_1}<0$, by Lemma~\ref{lem1} we have
$a_m^{(m)}>0$.
Moreover,  $C^0(0)=1$, therefore
$$
-\frac{a_m^{(m)}}{2^m}+
\prod\limits_{j=1}^m\left(1-\cos t_j\right)\cdot\sum_{j=m}^{n-m}a_j^{(m)}=1,
\quad
\sum_{j=m}^{n-m}a_j^{(m)}=\frac{1+\frac{a_m^{(m)}}{2^m}}{\prod\limits_{j=1}^m\left(1-\cos t_j\right)}>0.
$$
Let us build the following auxiliary polynomials
$$
S\left(\theta_1,\dots,\theta_m;t\right)=
N\left(\theta_1,\dots\theta_m\right)\cdot
\prod_{j=1}^m\left(\cos t-\cos \theta_j\right)\sum_{j=m}^{n-m}a_j^{(m)}\sin jt,
$$
$$
C\left(\theta_1,\dots,\theta_m;t\right)=
N\left(\theta_1,\dots\theta_m\right)\cdot
\left(-\frac{a_m^{(m)}}{2^m}+
\prod_{j=1}^m\left(\cos t-\cos \theta_j\right)\sum_{j=m}^{n-m}a_j^{(m)}\cos jt\right),
$$
where the normalizing factor $N\left(\theta_1,\dots,\theta_m\right)$ guaranties that the total sums of the coefficients $S\left(\theta_1,\dots,\theta_m;t\right)$ and $C\left(\theta_1,\dots,\theta_m;t\right)$ is equal to one. For the polynomial  $S\left(\theta_1,\dots,\theta_m;t\right)$ the set of sign changes will be
$T_\theta=\left\{\theta_1,\dots,\theta_m,t_{m+1},\dots,t_q\right\}$. It is clear that
$S\left(t_1,\dots,t_m;t\right)\equiv S^0(t)$ and
$C\left(t_1,\dots,t_m;t\right)\equiv C^0(t)$.
The factor $N\left(\theta_1,\dots,\theta_m\right)$
is determined by the condition
$C\left(\theta_1,\dots,\theta_m;0\right)=1$, i.e.
$$
N\left(\theta_1,\dots,\theta_m\right)=
\frac1{-\frac{a_m^{(m)}}{2^m}+\prod\limits_{j=1}^m\left(1-\cos\theta_j\right)
\sum\limits_{j=m}^{n-m}a_j^{(m)}}.
$$
Finally, the polynomials
$S\left(\theta_1,\dots,\theta_m;t\right)$ and
$C\left(\theta_1,\dots,\theta_m;t\right)$
could be defined by the expressions
$$
S\left(\theta_1,\dots,\theta_m;t\right)=
\frac{\prod\limits_{j=1}^m\left(\cos t-\cos \theta_j\right)\sum\limits_{j=m}^{n-m}a_j^{(m)}\sin jt}%
{-\frac{a_m^{(m)}}{2^m}+\prod\limits_{j=1}^m\left(1-\cos\theta_j\right)
\sum\limits_{j=m}^{n-m}a_j^{(m)}},
$$
$$
C\left(\theta_1,\dots,\theta_m;t\right)=
\frac{-\frac{a_m^{(m)}}{2^m}+
\prod\limits_{j=1}^m\left(\cos t-\cos \theta_j\right)\sum\limits_{j=m}^{n-m}a_j^{(m)}\cos jt}%
{-\frac{a_m^{(m)}}{2^m}+\prod\limits_{j=1}^m\left(1-\cos\theta_j\right)\sum\limits_{j=m}^{n-m}a_j^{(m)}}.
$$
Let us show that the value of $\rho_1$ for the pair
$\left\{S\left(\theta_1,\dots,\theta_m;t\right),\,
C\left(\theta_1,\dots,\theta_m;t\right)\right\}$
is bigger then for the pair  $\left\{S^0(t),\,C^0(t)\right\}$,
i.e. the pair $\left\{S^0(t),\,C^0(t)\right\}$  cannot be optimal.
\medskip

By the Corollary~\ref{cor3} we get
$$
C\left(\theta_1,\dots,\theta_m;\theta_1\right)=\dots
=C\left(\theta_1,\dots,\theta_m;\theta_m\right)=
-\frac{\frac{a_m^{(m)}}{2^m}}%
{-\frac{a_m^{(m)}}{2^m}+\prod\limits_{j=1}^m\left(1-\cos\theta_j\right)\sum\limits_{j=m}^{n-m}a_j^{(m)}}.
$$
Since  $a_m^{(m)}>0$ and $\sum\limits_{j=m}^{n-m}a_j^{(m)}>0$ ,
then all  $C\left(\theta_1,\dots,\theta_m;\theta_j\right)$, $j=1,\dots,m$,
are increasing functions of the parameters $\theta_1,\dots,\theta_m$.

From the continuity of trigonometric polynomials on $t$ and an all coefficients follows that for a small enough  $\varepsilon$ there is $\delta$, such that the conditions $0<\theta_j-t_j<\delta$, $j=1,\dots,m$,
imply $C\left(\theta_1,\dots,\theta_m;\theta_j\right)>C^0\left(t_j\right)$,
$j=1,\dots,m$,
$\left|C\left(\theta_1,\dots,\theta_m;t_j\right)-C^0\left(\theta_j\right)\right|<\varepsilon$,
$j=m+1,\dots,q$,
$\left|C\left(\theta_1,\dots,\theta_m;\pi\right)-C^0(\pi)\right|<\varepsilon$.
The above inequalities mean that the value of
$$
\min\left\{C\left(\theta_1,\dots,\theta_m;\theta_1\right),\dots,
C\left(\theta_1,\dots,\theta_m;\theta_m\right),
C\left(\theta_1,\dots,\theta_m;t_{m+1}\right),\dots,\right.
$$
$$\left.
C\left(\theta_1,\dots,\theta_m;t_q\right),
C\left(\theta_1,\dots,\theta_m;\pi\right)\right\}
$$
is larger then
$\min\left\{C^0\left(t_1\right),\dots,C^0\left(t_q\right),C^0(\pi)\right\}$
at least for small enough positive
$\theta_j-t_j$, $j=1,\dots,m$, i.e. the pair
$\left\{S^0(t),\,C^0(t)\right\}$  is not optimal one.

{\it Case 2.}
By Corollary~\ref{lem4},
$$
C^0(\pi)=-\frac{a_m^{(m)}}{2^m}-\prod_{j=1}^m\left(1+\cos t_j\right)
\sum_{j=m}^{n-m}(-1)^ja_j^{(m)},
$$
$$
C\left(\theta_1,\dots,\theta_m;\pi\right)=
-\frac{\frac{a_m^{(m)}}{2^m}+\prod\limits_{j=1}^m\left(1+\cos\theta_j\right)
\sum\limits_{j=m}^{n-m}(-1)^ja_j^{(m)}}%
{-\frac{a_m^{(m)}}{2^m}+\prod\limits_{j=1}^m\left(1-\cos\theta_j\right)
\sum\limits_{j=m}^{n-m}a_j^{(m)}},
$$
and
$$C\left(t_1,\dots,t_m;\pi\right)=C^0(\pi).$$

Since we assume that  $C^0(\pi)\le-\frac{a_m^{(m)}}{2^m}$,
then   $\sum\limits_{j=m}^{n-m}(-1)^ja_j^{(m)}\ge0$,
and the quantity
$
\frac{a_m^{(m)}}{2^m}+\prod\limits_{j=1}^m\left(1+\cos \theta_j\right)
\sum\limits_{j=m}^{n-m}(-1)^ja_j^{(m)}
$
is decreasing with respect to each parameter  $\theta_1,\dots,\theta_m$. For small increments of
$\theta_j-t_j$, $j=1,\dots,m$ the quantity
$
-\frac{a_m^{(m)}}{2^m}+\prod\limits_{j=1}^m\left(1-\cos \theta_j\right)
\sum\limits_{j=m}^{n-m}a_j^{(m)}
$
is close to 1 and is increasing with respect to each parameter $\theta_1,\dots,\theta_m$.
Therefore $C\left(\theta_1,\dots,\theta_m;\pi\right)$ is increasing with respect to each parameters $\theta_1,\dots,\theta_m$.
At the same time by Lemma 5,
$   C\left(\theta_1,\dots,\theta_m;\theta_j\right)$, $j=1,\dots,m$, are equal and are increasing
with respect to each parameters  $\theta_1,\dots,\theta_m$. Therefore in this case as well the pair
$\left\{S^0(t),C^0(t)\right\}$ cannot be an optimal one.
\medskip

Hence, it is shown that the set $T$ is empty, i.e. for the optimal pair
$\left\{S^0(t),\,C^0(t)\right\}$ we have $S^0(t)\ge0$, $t\in[0,\pi]$.

The trigonometric polynomial $S^0(t)$ can be written as
$$
S^0(t)=\sin t\cdot\left(\gamma_1+2\gamma_2\cos t+\dots+2\gamma_n\cos(n-1)t\right),
$$
where
$\gamma_s=\sum_{s\le j\le n} a_j$, and the summation runs on indexes  $j$ of same parity with  $s$, $s=1,\dots,n$.

There is a bijection between $a_1,\dots,a_n$  and $\gamma_1,\dots,\gamma_n$. The normalization condition $\sum\limits_{j=1}^na_j=1$  is equivalent to $\gamma_1+\gamma_2=1$.

Since  $T=\emptyset$, then
$$
\overline{\rho_1}=
\max_{\left(a_1,\dots,a_n\right)\in A_R}\{C(\pi)\}=
\max_{\left(a_1,\dots,a_n\right)\in A_R}\left\{\sum\limits_{j=1}^n(-1)^ja_j\right\}.
$$
Note that  $\sum\limits_{j=1}^n(-1)^ja_j=-\gamma_1+\gamma_2$.

The polynomial $\frac{S^0(t)}{\sin t}$  is non-negative and the well-known Fejer inequality for non-negative polynomials ~\cite{2}
(see also~\cite[6.7, Problem 52]{3}) implies that
$$
\left|\gamma_2\right|\le\cos\frac\pi{n+1}\cdot\left|\gamma_1\right|.
$$
Then
$$
\overline{\rho_1}\le\overline{\rho_2}=\max_{\gamma_1,\gamma_2}\left\{-\gamma_1+\gamma_2:\ \gamma_1+\gamma_2=1,\,
\left|\gamma_2\right|\le\cos\frac\pi{n+1}\cdot\left|\gamma_1\right|\right\}.
$$
The conditional maximum is achieved for
$$
\gamma_1^0=\frac1{1+\cos\frac\pi{n+1}},\quad
\gamma_2^0=\frac{\cos\frac\pi{n+1}}{1+\cos\frac\pi{n+1}},
$$
and is equal to
$$
\overline{\rho_2}=-\frac{1-\cos\frac\pi{n+1}}{1+\cos\frac\pi{n+1}}=-\tan^2\frac\pi{2(n+1)}.
$$
Since $\gamma_1^0$ and $\gamma_2^0$  change the Fejer inequality to the equality then there exist a (unique) nonnegative polynomial
with the the zero and first coefficients $\gamma_1^0$ and $\gamma_2^0$ correspondently. Therefore $\overline{\rho_1}=\overline{\rho_2}.$\\

Therefore, in the polynomial $$\frac{S^0(t)}{\sin t}=\gamma_1^0+2\gamma_2^0\cos t+\dots+2\gamma_n^0\cos(n-1)t$$  all its coefficients are determined in a unique way, thus the coefficients  $a_1^0,\dots,a_n^0$  are determined in a unique way as well:
$a_1^0=\gamma_1^0-\gamma_3^0$,
$a_2^0=\gamma_2^0-\gamma_4^0$,
$a_3^0=\gamma_3^0-\gamma_5^0,\dots$, $a_{n-2}^0=\gamma_{n-2}^0-\gamma_n^0$,
$a_{n-1}^0=\gamma_{n-1}^0$,
$a_n^0=\gamma_n^0$.\\

Since the $a_j^0$  are positive
$$
\sum\limits_{j=1}^n\left|a_j^0\right|=\sum\limits_{j=1}^na_j^0=1
$$
and the sum is independent on  $R$. This means that  for all  $a_1,\dots,a_n$, such that
$\sum\limits_{j=1}^n\left|a_j\right|=1$ the following inequality is valid:
$$
\rho_1\left(a_1,\dots,a_n\right)\le
\overline{\rho_1},\quad
\rho\left(a_1,\dots,a_n\right)\le
\overline{\rho}\le\overline{\rho_1}.
$$

To show that for the function $\rho\left(a_1,\dots,a_n\right)$ the upper bound is equal to
$\overline{\rho_1}$, let us consider a one-parameter family of trigonometric polynomials
$$
S^\varepsilon(t)=\frac{a_1^0+\varepsilon}{1+\varepsilon}\sin t+
\frac{a_2^0}{1+\varepsilon}\sin 2t+\dots+
\frac{a_n^0}{1+\varepsilon}\sin nt.
$$
It is clear that
$\frac{a_1^0+\varepsilon}{1+\varepsilon}+
\frac{a_2^0}{1+\varepsilon}+\dots+
\frac{a_n^0}{1+\varepsilon}=1$  and
$S^\varepsilon(t)=\frac{S^0(t)}{1+\varepsilon}+\frac\varepsilon{1+\varepsilon}\sin t$,
$C^\varepsilon(t)=\frac{C^0(t)}{1+\varepsilon}+\frac\varepsilon{1+\varepsilon}\cos t$.
\medskip

Now, for all $t\in(0,\pi)$ and $\varepsilon>0$ we have $S^\varepsilon(t)>0$.
Since $C^\varepsilon(\pi)=\frac1{1+\varepsilon}\overline{\rho_1}-\frac\varepsilon{1+\varepsilon}$,
then $C^\varepsilon(\pi)<\overline{\rho_1}$  and
$C^\varepsilon(\pi)\to\overline{\rho_1}$  when $\varepsilon\to0$.
\medskip

The above conditions and independance of the coefficients on $R$ means that
$$
I=\overline{\rho}=\overline{\rho_1}=
\sup_{a_1,\dots,a_n}\left\{\rho\left(a_1,\dots,a_n\right)\right\}
=-\tan^2\frac\pi{2(n+1)}.
$$
 \end{proof}

 \begin{corollary}\label{cor5}
Let a pair of conjugate trigonometric polynomials
$$
C(t)=\sum_{j=1}^na_j\cos jt,\quad
S(t)=\sum_{j=1}^na_j\sin jt,
$$
be normalized by the condition
$\sum\limits_{j=1}^na_j=1$.

Let $I$  be a solution of extremal problem
$$\max\limits_{a_1,\dots,a_n}\min\limits_t\{C(t):\ T\cup\{\pi\}\},$$
where $T$  is a set of sign changes for the function $S(t)$ on$(0,\pi)$.

Then there exists a unique pair of polynomials
$$C^0(t)=\sum\limits_{j=1}^na_j^0\cos jt,\qquad S^0(t)=\sum\limits_{j=1}^na_j^0\sin jt,$$
where
$a_j^0=2\cdot\tan\frac\pi{2(n+1)}\cdot\left(1-\frac j{n+1}\right)\cdot\sin\frac{\pi j}{n+1}$,
$j=1,\dots,n$ which produces the solution. Moreover,
$I=-\tan^2\frac\pi{2(n+1)}$.
\end{corollary}

\begin{proof}
To prove the corollary it is sufficient to find the coefficients
$a_1^0,\dots,a_n^0$.
The polynomial
$\frac{S^0(t)}{\sin t}$  is proportional to the Fejer polynomial
$$
\frac{S^0(t)}{\sin t}=
\frac1{1+\cos\frac\pi{n+1}}+
\frac{2\cos\frac\pi{n+1}}{1+\cos\frac\pi{n+1}}\cos t+\dots
$$
$$
=\frac{1-\cos\frac\pi{n+1}}{n+1}\cdot
\frac{2\cos^2\frac{n+1}2t}{\left(\cos t-\cos\frac\pi{n+1}\right)^2}=
\gamma_1^0+2\gamma_2^0\cos t+\dots+2\gamma_n^0\cos(n-1)t.
$$
From there the coefficients $\gamma_1^0,\dots,\gamma_n^0$
are defined by the following rules (see~\cite{5})
$$
\gamma_j^0=
\frac1{2(n+1)\sin\frac\pi{n+1}\cdot\left(1+\cos\frac\pi{n+1}\right)}\cdot
\left((n-j+3)\sin\frac{\pi j}{n+1}-(n-j+1)\sin\frac{\pi(j-2)}{n+1}\right).
$$
Therefore,
\begin{multline}\label{(8)}
a_j^0=\gamma_j^0-\gamma_{j+2}^0=
2\cdot\tan\frac\pi{2(n+1)}\cdot
\left(1-\frac j{n+1}\right)\cdot\sin\frac{\pi j}{n+1},
\quad
j=1,\dots,n
\end{multline}
where $\gamma_{n+1}^0=\gamma_{n+2}^0=0$.
\end{proof}

The obtained above results can be applied for the development of the methods of optimal control
of chaos for the families of discrete autonomous systems with the delayed feedback control  (DFC - methods)~\cite{4}.

Let us consider non-closed scalar nonlinear discrete system of the following type
\begin{equation}\label{(10)}
x_{n+1}=f\left(x_n\right),\ x_n\in\mathbb R^1,\ n=1,2,\dots,
\end{equation}
which does have an unstable equilibrium  $x^*$. Moreover let us assume that the differentiable function $f$ depends on finite number of parameters and that for each admissible set of those parameters the function is defined on a some bounded interval and maps it into itself. In this case the equilibrium  $x^*$ and the multiplier  $\mu=f'\left(x^*\right)$ are dependent on those parameters. Assume now that  $\mu\in\left(-\mu^*,-1\right)$, $\mu^*>1$.
It is needed to stabilize the equilibrium by the control of the following type

\begin{equation}\label{(11)}
u=\sum_{j=1}^N\varepsilon_jf\left(x_{n-j+1}\right),\quad
\sum_{j=1}^N\varepsilon_j=0,\ \left|\varepsilon_j\right|<1,\ j=1,\dots,N,
\end{equation}
in a such way that the depth of the used prehistory with the delayed feedback $N^*=N-1$ is minimal.

 \begin{theorem}\label{teo2}

Let the system ~(\ref{(10)}) have non-stable equilibrium with multiplier
$\mu\in\left(-\mu^*,-1\right)$, $\mu^*>1$. Then there exists a control of the type
  ~(\ref{(11)}), that stabilizes the equilibrium that is optimal
relative to the minimum depth of prehistory with delayed feedback. Moreover,
$$
N^*=\left[\frac\pi{2\cdot\arccot\sqrt{\mu^*}}\right]-1.
$$
 \end{theorem}

 \begin{proof}
The system
 ~(\ref{(10)}), closed by the control
  ~(\ref{(11)}),
can be represented in the form
 \begin{equation}\label{(12)}
x_{n+1}=\left(1+\varepsilon_1\right)f\left(x_n\right)+
\sum_{j=2}^N\varepsilon_jf\left(x_{n-j+1}\right).
\end{equation}
The equilibrium  $x^*$  of the non-closed system~(\ref{(10)}) is still 
an equilibrium for the unclosed system ~(\ref{(12)}) as well. The multipliers of the system
 ~(\ref{(12)}) satisfy the characteristic equation
 \begin{equation}\label{(13)}
\lambda^N+|\mu|\left(\left(1+\varepsilon_1\right)\lambda^{N-1}+
\varepsilon_2\lambda^{N-2}+\dots+\varepsilon_N\right)=0,
\end{equation}
where
$\left(1+\varepsilon_1\right)+\varepsilon_2+\dots+\varepsilon_N=1$.\\

Once again, on a unit circle
$$
1+|\mu|\left(\left(1+\varepsilon_1\right)e^{-it} +
\varepsilon_2e^{-2it}+\dots+\varepsilon_Ne^{-iNt}\right)=0,
$$
or
$$
|\mu|\Re\left(\left(1+\varepsilon_1\right)e^{-it} +
\varepsilon_2e^{-2it}+\dots+\varepsilon_Ne^{-iNt}\right)=-1,
$$
while
$$
\Im\left(\left(1+\varepsilon_1\right)e^{-it} +
\varepsilon_2e^{-2it}+\dots+\varepsilon_Ne^{-iNt}\right)=0.
$$
Now, if
$$
|\mu|\min_{t\in[0,\pi]}\Re\left(\left(1+\varepsilon_1\right)e^{-it} +
\varepsilon_2e^{-2it}+\dots+\varepsilon_Ne^{-iNt}\right)>-1,
$$
then for the given choice of $\varepsilon_i$ there are no roots on the unit circle for the given value $|\mu|$ as well as for any smaller one. Therefore all the roots are inside the unit disc which guaranties the stability.\\

So, if
$$
|\mu|\sup_{a_1+\dots+a_N}\left(\min_{t\in[0,\pi]}\Re\left(\left(1+\varepsilon_1\right)e^{-it} +
\varepsilon_2e^{-2it}+\dots+\varepsilon_Ne^{-iNt}\right)\right)>-1,
$$
there there exits a choice of $\varepsilon_i$  which guaranties the stability.\\

The above statement can be written in the form $|\mu|I>-1$ or
$$|\mu||I|<1.$$
Since $|\mu|<\mu^*\le\cot^2\frac\pi{2(N+1)}$ the choice $N>-1+\frac\pi{2\cdot\arccot\sqrt{\mu^*}}$ implies
$|\mu||I|<1$ and we get the stability with the depth of the used prehistory $N^*=\left[\frac\pi{2\cdot\arccot\sqrt{\mu^*}}\right]-1$.\\

The coefficients of the strengthening $\varepsilon_1,\dots,\varepsilon_N$
    for the optimal control are be determined by the equalities
$1+\varepsilon_1^0=\alpha_1^0$,
$\varepsilon_2^0=\alpha_2^0,\dots,$
$\varepsilon_N^0=\alpha_N^0$, where
$\alpha_1^0,\dots,\alpha_N^0$  are given by  ~(\ref{(8)}).
 Since   $0<\alpha_j^0<1$, $j=1,\dots,N$, then
$-1<\varepsilon_1^0<0,$ $0<\varepsilon_j^0<1$, $j=2,\dots,N$.
\end{proof}

\begin{remark}\label{rem2}
The quantity  $\sum\limits_{j=1}^N\left|\varepsilon_j\right|=2\left(1-\alpha_1^0\right)<2$,
and so is bounded by a constant independent of $\mu^*$ and $N$.
\end{remark}

\begin{example}
For an one-parametric logistic map we have
$$
f(x)=h\cdot x\cdot(1-x),\quad0\le h\le4,
$$
$f:\ [0,1]\to[0,1]$. For  $h\in(3,4]$ the equilibrium  $x^*=1-\frac1h$ is not stable and
$\mu\in[-2,-1)$.

Then
$\frac\pi{2\cdot\arccot\sqrt2}\approx 2,55$, therefore the depth of prehistory with the delayed feedback is $N^*=1$; the optimal coefficients of the strength $\varepsilon_1^0=-\frac13$, $\varepsilon_2^0=\frac13$,
 and the desire control  is $u=-\frac13f\left(x_n\right)+\frac13f\left(x_{n-1}\right)$.
\end{example}

\centerline{\bf Acknowledgment}
\medskip

The authors are grateful to A.A.~Korenovskii, A.A.~Soljanik and A.M.~Stokolos for a valuable ideas suggested in the process of discussion of the problems from the paper.

Odessa National Polytechnic University, 1 Shevchenko Ave., Odessa 65044, Ukraine.

Email: {\tt dmitrishin@opu.ua}

\end{document}